\documentclass{amsart}
\usepackage{amssymb,amsmath,latexsym}
\usepackage{times}
\usepackage{graphicx}
\newtheorem{thm}{Theorem}[section]
\newtheorem{lem}[thm]{Lemma}
\newtheorem{cor}[thm]{Corollary}

\numberwithin{equation}{section}

\newcommand{\co}{\hbox{\rm co}}
\newcommand{\si}{\hbox{\rm si}}

\title[Representations of bicircular lift matroids]
{Representations of bicircular lift matroids}

\date{\today}

\author[Rong Chen and Zifei Gao]{Rong Chen and Zifei Gao}
\address{Center for Discrete Mathematics, Fuzhou University,
Fuzhou, P. R. China.}

\thanks{Email address: rongchen@fzu.edu.cn (Rong Chen).\\
\ \ \ This research was supported by NSFC (No.11201076). The first author was also partially supported  by NSFC (No. 11471076).}

\begin{document}
\begin{abstract}
Bicircular lift matroids are a class of matroids defined on the edge set of a graph. For a given graph $G$, the circuits of its bicircular lift matroid are the edge sets of those subgraphs of $G$ that contain at least two cycles, and are minimal with respect to this property. The main result of this paper is a characterization of when two graphs give rise to the same bicircular lift matoid, which answers a question proposed by Irene Pivotto. In particular, aside from some appropriately defined ``small'' graphs, two graphs have the same bicircular lift matroid if and only if they are $2$-isomorphic in the sense of Whitney. \\

\end{abstract}



\maketitle

\section{Introduction}
We assume the reader is familiar with fundamental definitions in matroid and graph theory. For a graph $G$, a set $X\subseteq E(G)$ is a {\sl cycle} if $G|X$ is a connected 2-regular graph. Bicircular lift matroids are a class of matroids defined on the edge set of a graph. For a given graph $G$, the circuits of its bicircular lift matroid $L(G)$ are the edge sets of those subgraphs of $G$ that contain at least two cycles, and are minimal with respect to this property. That is, the circuits of $L(G)$ consists of the edge sets of two edge-disjoint cycles with at most one common vertex, or three internally disjoint paths between a pair of distinct vertices. Bicircular lift matroids are a special class of lift matroids that arises from biased graphs. Biased graphs and lift matroids were introduced by Zaslavsky in \cite{Zas89,Zas91}.

Whitney \cite{Whitney} characterized which graphs have isomorphic graphic matroids. Chen, DeVos, Funk and Pivotto \cite{CDFP} generalized Whitney's result and characterized which biased graphs have isomorphic graphic frame matroids. Matthews \cite{Matthews} characterized which graphs give rise to isomorphic bicircular matroids that are graphic. Coullard, del Greco and Wagner \cite{CGW,Wagner} characterized which graphs give rise to isomorphic bicircular matroids. In this paper, we characterize which graphs give rise to isomorphic bicircular lift matroids, which answers a question proposed by Pivotto in the Matroid Union blog \cite{Irene}. In particular, except for some special graphs, each of which is a subdivision of a graph on at most four vertices, two graphs have the same bicircular lift matroid if and only if they are $2$-isomorphic in the sense of Whitney \cite{Whitney}. The main result is used in \cite{Chen16} to prove that the class of matroids that are graphic or bicircular lift has a finite list of excluded minors.

To state our result completely we need more definitions. Let $k,l,m$ be positive integers. We denote by $K_m$ the complete graph with $m$ vertices. We denote by $K_2^m$ the graph obtained from $K_2$ with its unique edge replaced by $m$ parallel edges. And we denote by $K_3^{k,l,m}$ the graph obtained from $K_3$ with its three edges replaced by $k,l,m$ parallel edges respectively. A graph obtained from graph $G$ by replacing some edges of $G$ with internally disjoint paths is a {\sl subdivision} of $G$. Note that $G$ is a subdivision of itself. A path $P$ of a connected graph $G$ is an {\sl ear} if each internal vertex of $P$ has degree two and each end-vertex has degree at least three in $G$, and $P$ is contained in a cycle. A graph $G$ is {\sl $2$-edge-connected} if each edge of $G$ is contained in some cycle. Let $M(G)$ denote the graphic matroid of a graph $G$.


Given a set  $X$ of edges, we let $G|X$ denote the subgraph of $G$ with edge set $X$ and no isolated vertices. Let $(X_1,X_2)$ be a partition of $E(G)$ such that $V(G|X_1)\cap V(G|X_2)=\{u_1,u_2\}$. We say that $G'$ is obtained by a {\sl Whitney Switching} on $G$ on $\{u_1,u_2\}$ if $G'$ is a graph obtained by identifying vertices $u_1,u_2$ of $G|X_1$ with vertices $u_2,u_1$ of $G|X_2$, respectively. A graph $G'$ is {\sl 2-isomorphic} to $G$ if $G'$ is obtained from $G$ by a sequence of the operations: Whitney switchings, identifying two vertices from distinct components of a graph, or partitioning a graph into components each of which is a block of the original graph.

\begin{thm}(Whitney's\ $2$-Isomorphism\ Theorem)\label{Whitney}
Let $G_1$ and $G_2$ be graphs. Then $M(G_1)\cong M(G_2)$ if and only if $G_1$ and $G_2$ are $2$-isomorphic.
\end{thm}

It follows from Theorem \ref{Whitney} that if $G_1$ and $G_2$ are $2$-isomorphic, then $L(G_1 )= L(G_2 )$. The converse, however, is not true. This can be seen by choosing $G_1$ and $G_2$ to be isomorphic to $K_4$, but not to each other. Much of the remainder of the paper is aimed at characterizing when the converse to this statement is not true.

Let $G_1$ and $G_2$ be graphs with $L(G_1 )= L(G_2 )$. Since $E(G_i)$ is independent in $L(G_i)$ if and only if $G_i$ has at most one cycle, we may assume that $G_1$ and $G_2$ have at least two cycles. Moreover, since $e$ is a cut-edge of $G_1$ if and only if $e$ is a cut-edge of $G_2$ or $G_2\backslash e$ is a forest, an edge is a cut-edge of $G_1$ if and only if it is a cut-edge of $G_2$. Hence,  to simplify the analysis below, it will be assumed for the remainder of the paper that $G_1$ and $G_2$ are 2-edge-connected. Observe that when $L(G_1 )$ has only one circuit, it is straightforward to characterize the structure of both $G_1$ and $G_2$. Thus, the remainder of the paper will further restrict the analysis to the case that $L(G_1 )$ has at least two circuits. In the paper, we prove

\begin{thm}\label{1-con}
Let $G_1$ be a $2$-edge-connected graph such that $L(G_1)$ contains at least two circuits. Let $G_2$ be a graph with $L(G_1 )=L(G_2 )$. Then at least one of the following holds.
\begin{itemize}
    \item[(1)] $G_1$ and $G_2$ are $2$-isomorphic.
    \item[(2)] $G_1$ and $G_2$ are $2$-isomorphic to subdivisions of $K_4$, where the edge set of an ear of $G_1$ is also the edge set of an ear of $G_2$.
    \item[(3)] $G_1$ and $G_2$ are $2$-isomorphic to subdivisions of $K_3^{m,2,n}$ for some $m\in\{1,2\}$ and $n\geq2$, where the edge set of an ear of $G_1$ is also the edge set of an ear of $G_2$. Moreover, when $n\geq3$, the $n$ ears in $G_1$ having the same ends also have the same ends in $G_2$. 
    \item[(4)] $G_1$ and $G_2$ are $2$-isomorphic to the graphs pictured in Figure \ref{Figure 1} .
\end{itemize}
\end{thm}

 \begin{figure}[htbp]
\begin{center}
\includegraphics[page=1,height=10cm]{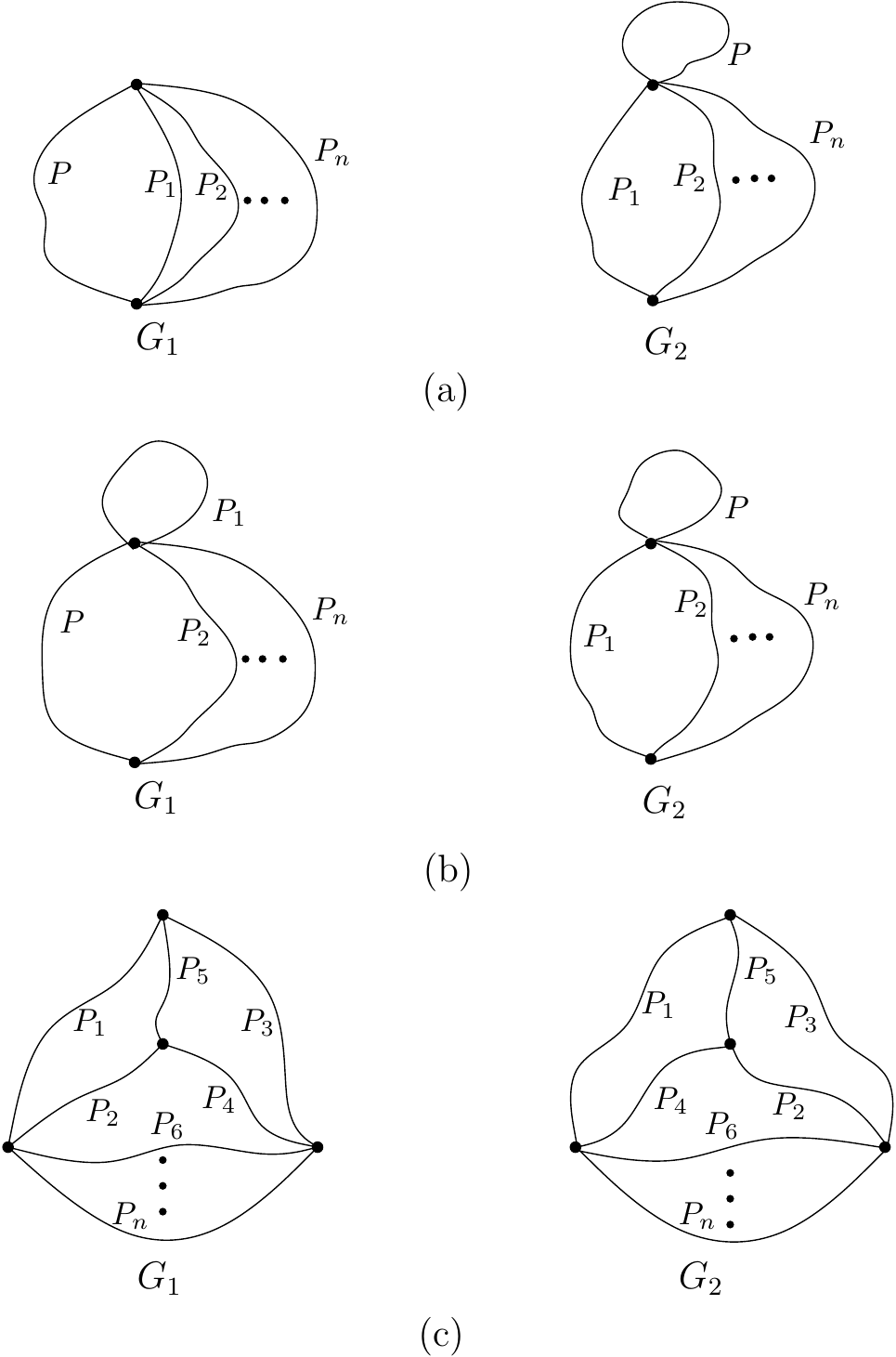}
\caption{In (a) and (b), $n\geq3$; and in (c), $n\geq2$}
\label{Figure 1}
\end{center}
\end{figure}

The following result, which is an easy consequences of Theorem \ref{1-con}, is used in \cite{Chen16} to prove that the class of matroids that are graphic or bicircular lift matroids has a finite list of excluded minors.

Two elements are a {\sl series pair} of a graph $G$ if and only if each cycle can not intersect them in exactly one element.  A {\sl series class} is a maximal set $X\subseteq E(G)$ such that every two edges of $X$ form a series pair. Let $\co(G)$ denote a graph obtained from $G$ by contracting all cut-edges from $G$ and then, for each series class $X$, contracting all but one distinguished element of $X$.

\begin{cor}\label{1-con+}
Let $G_1$ and $G_2$ be connected graphs with $L(G_1)=L(G_2)$ and such that $L(G_1)$ has at least two circuits. If $|V(\co(G_1))|\geq5$ then $G_1$ and $G_2$ are $2$-isomorphic.
\end{cor}




\section{Proof of Theorem \ref{1-con}.}
Let $G$ be a graph, and $e,f\in E(G)$. We say that $e$ is a {\sl link} if it has distinct end-vertices; otherwise $e$ is a {\sl loop}. If $\{e,f\}$ is a cycle, then $e$ and $f$ are {\sl parallel}. A {\sl parallel class} of $G$ is a maximal subset $P$ of $E(G)$ such that any two members of $P$ are parallel and no member is a loop. Moreover, if $|P|\geq2$ then $P$ is {\sl non-trivial}; otherwise $P$ is {\sl trivial}. Let $\si(G)$ denote the graph obtained from $G$ by deleting all loops and all but one distinguished element of each non-trivial parallel class. Obviously, the graph we obtain is uniquely determined up to a renaming of the distinguished elements. If $G$ has no loops and no non-trivial parallel class, then $G$ is {\sl simple}.


The following result is implied in (\cite{Zas91}, Theorem 3.6.).

\begin{lem}\label{contract-L}
Let $e$ be an edge of a graph $G$. Then we have
 \begin{itemize}
    \item[(1)] $L(G\backslash e)=L(G)\backslash e$;
    \item[(2)] when $e$ is a loop, $L(G)/e=M(G\backslash e)$;
    \item[(3)] when $e$ is a link, $L(G)/e=L(G/e)$.
 \end{itemize}
\end{lem}

Corollary \ref{loop-iso} follows immediately from Lemma \ref{contract-L} (2) and Theorem \ref{Whitney}. 

\begin{cor}\label{loop-iso}
Let $G_1, G_2$ be graphs with $L(G_1)=L(G_2)$, and $e$ a loop of both $G_1$ and $G_2$. Then $G_1$ and $G_2$ are $2$-isomorphic.
\end{cor}

The idea used to prove the following Lemma was given by the referee.

\begin{lem}\label{contract-G}
Let $G_1$ and $G_2$ be connected graphs without loops and with $|V(G_1)|=|V(G_2)|$ and $E(G_1)=E(G_2)$. Assume that for each edge $e\in E(G_1)$ the graphs $G_1/e$ and $G_2/e$ are $2$-isomorphic. Then $G_1$ and $G_2$ are $2$-isomorphic.
 \begin{proof}
 By Whitney's 2-Isomorphism Theorem, to prove the result it suffices to show that each spanning tree of $G_1$ is also a spanning tree of $G_2$. Let $T_1$ be a spanning tree of $G_1$, and let $T_2$ be the subgraph of $G_2$ induced by $E(T_1)$. Assume that $T_2$ is not a spanning tree of $G_2$. Since $|V(G_1)|=|V(G_2)|$, the subgraph $T_2$ contains a cycle $C$. Let $e$ be an edge in $E(T_1)$. Then $T_1/e$ is acyclic and $T_2/e$ is not,  and so $G_1/e$ and $G_2/e$ are not $2$-isomorphic; a contradiction.
 \end{proof}
\end{lem}

\begin{lem}\label{loop-case}
Let $G_1$ be a $2$-edge-connected graph such that $L(G_1)$ contains at least two circuits. Let $G_2$ be a graph with $L(G_1)=L(G_2)$. Assume that $G_1$ has a link $e$ such that $e$ is a loop of $G_2$. Then $G_1$ and $G_2$ are $2$-isomorphic to the graphs pictured in Figure \ref{Figure 3}.

\begin{figure}[htbp]
\begin{center}
\includegraphics[page=2,height=6.5cm]{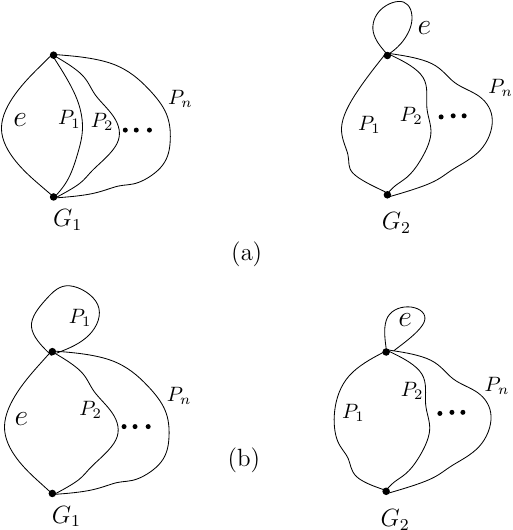}
\caption{$n\geq3$}
\label{Figure 3}
\end{center}
\end{figure}

 \begin{proof}
Since $L(G_1)$ contains at least two circuits and $L(G_1)=L(G_2)$, the graph $G_2-\{e\}$ has cycles $C_1$ and $C_2$ such that $C_1\cup C_2$ is a circuit of $L(G_2)$. Since $e$ is a loop of $G_2$, for some integer $k\in \{2,3\}$ there is a partition $(P_1,P_2,\cdots,P_k)$ of $E(C_1\cup C_2)$ such that when $k=2$ the sets $P_1\cup \{e\}$ and $P_2\cup \{e\}$ are circuits of $L(G_1)$, and when $k=3$ the sets $P_1\cup P_2\cup \{e\}$, $P_2\cup P_3\cup \{e\}$ and $P_1\cup P_3\cup \{e\}$ are circuits of $L(G_1)$. Since $E(C_1\cup C_2)$ is also a circuit of $L(G_1)$ and $e$ is a link of $G_1$, it is easy to verify that $k=3$ (that is, $C_1\cup C_2$ is a theta-subgraph of $G_2$.) and  {\bf (1)}  $G_1|C_1\cup C_2\cup \{e\}$ is 2-isomorphic to graphs pictured in Figure \ref{Figure 4}. Hence, by the arbitrary choice of $C_1$ and $C_2$, {\bf (2)} no two cycles in $G_2$ have at most one common vertex; and {\bf (3)} each ear of a theta-subgraph of $G_2$ is a cycle in $G_1$ or a path connecting the end-vertices of $e$ in $G_1$.

\begin{figure}[htbp]
\begin{center}
\includegraphics[page=3,height=2.5cm]{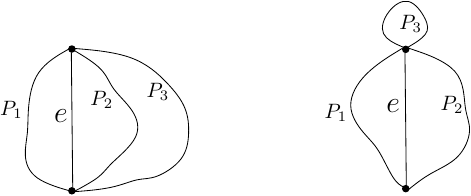}
\caption{Possible structures of graph $G_1|C_1\cup C_2\cup \{e\}$.}
\label{Figure 4}
\end{center}
\end{figure}

For each edge $f\in E(G_2)-(C_1\cup C_2\cup \{e\})$, there is a set $X$ with $f\in X\subseteq E(G_2)-(C_1\cup C_2\cup \{e\})$ such that $G_2|C_1\cup C_2\cup X$ is 2-edge-connected. By (2) $G_2|C_1\cup C_2\cup X$ is a subdivision of $K_4$ or $K_2^4$. (1) and (3) imply that $G_2|C_1\cup C_2\cup X$ is a subdivision of $K_2^4$. Repeating the process several times, we have that $G_2-\{e\}$ is a $K_2^n$-subdivision for some integer $n\geq3$. Hence, $G_1$ and $G_2$ are 2-isomorphic to the graphs pictured in Figure \ref{Figure 3}.
 \end{proof}
\end{lem}

By Lemma \ref{loop-case}, to prove Theorem \ref{1-con} we only need to consider the case that an edge is a link in $G_1$ if and only if it is a link in $G_2$.

\begin{lem}\label{ear-to-ear}
Let $G_1$ and $G_2$ be connected and $2$-edge-connected graphs with $L(G_1)=L(G_2)$ such that $L(G_1)$ has at least two circuits and such that each series class of $G_i$ is an ear of  $G_i$ for each $i\in\{1,2\}$. Then a set of edges is the edge set of an ear of $G_1$ if and only if it is the edge set of an ear of $G_2$.
 \begin{proof}
Assume otherwise. Without loss of generality assume that $e$ and $f$ are contained in some ear of $G_1$, but not in the some ear of $G_2$.  Evidently, $e$ is not in any cycle of $G_1-\{f\}$ and $L(G_1-\{f\})$ has a circuit as $L(G_1)$ has at least two circuits. Moreover, since $L(G_1-\{f\})=L(G_2-\{f\})$, the edge $e$ is a coloop of $G_2-\{f\}$; so $\{e,f\}$ is  a bond of $G_2$. Then $e$ and $f$ are contained in the some ear of $G_2$ as each series class of $G_2$ is an ear of $G_2$, a contradiction.
\end{proof}
\end{lem}

By possibly applying a sequence of Whitney's switching we can assume that each series class in a graph $G$ is an ear of $G$. Furthermore, by Lemma \ref{ear-to-ear} we can further assume that a set of edges is the edge set of an ear of $G_1$ if and only if it is the edge set of an ear of $G_2$. Hence, we only need consider cosimple graphs, where a graph is {\sl cosimple} if it has no cut-edges or non-trivial series classes.


Let $loop(G)$ be the set consisting of loops of $G$.


 \begin{figure}[htbp]
\begin{center}
\includegraphics[page=4,height=7cm]{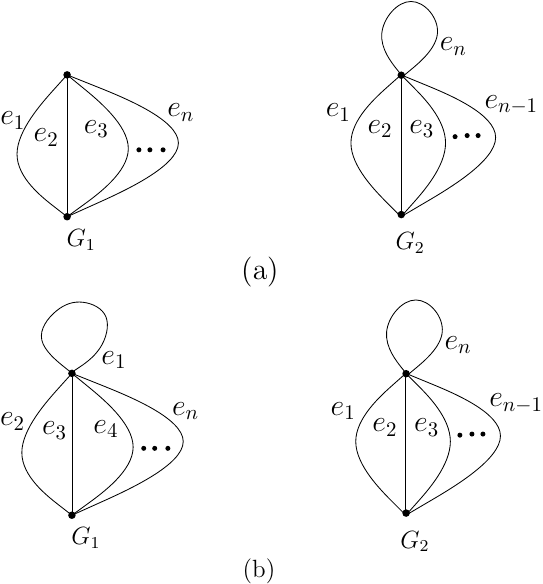}
\caption{$n\geq3$.}
\label{Figure 5}
\end{center}
\end{figure}

\begin{lem}\label{n<4}
Let $G_1$ and $G_2$ be cosimple $2$-edge-connected graphs with $2\leq |V(G_1)|=|V(G_2)|\leq 3$. Assume that $L(G_1)=L(G_2)$ and $L(G_1)$ contains at least two circuits. Then exactly one of the following holds.
 \begin{itemize}
   \item[(1)] $G_1$ and $G_2$ are $2$-isomorphic.
   \item[(2)] $|V(G_1)|=2$, the graphs $G_1$ and $G_2$ are isomorphic to the graphs pictured in Figure \ref{Figure 5}.
    \item[(3)] $G_1$ and $G_2$ are $2$-isomorphic to $K_3^{m,2,n}$ for some integers $m\in\{1,2\}$ and $n\geq2$, moreover, the $n$ parallel edges in $G_1$ are also the $n$ parallel edges in $G_2$ when $n\geq3$.
 \end{itemize}
 \end{lem}

 \begin{proof}
By Lemma \ref{loop-case} we may assume that $loop(G_1)=loop(G_2)$. Then the lemma holds when $|V(G_1)|=2$. So assume that  $|V(G_1)|=3$. Since $loop(G_1)=loop(G_2)$, each non-trivial parallel class of $G_1$ with at least three edges must be also a non-trivial parallel class of $G_2$. Hence, when $G_1$ has two parallel classes with at least three edges, (1) holds. So we may assume that $G_1$ has at most one parallel class with at least three edges. On the other hand, since $G_1$ and $G_2$ are cosimple, $G_1$ and $G_2$ have three parallel classes and at least two of them are non-trivial. Hence, when $G_1$ has no loops, (3) obviously holds; when $G_1$ has a loop, since $loop(G_1)=loop(G_2)$, Corollary \ref{loop-iso} implies that $G_1$ and $G_2$ are $2$-isomorphic, that is, (1) holds. 
 \end{proof}

The {\sl star} of a vertex $v$ in a graph $G$, denoted by $st_G(v)$, is the set of edges of $G$ incident with $v$.

\begin{lem}\label{n=4}
Let $G_1$ and $G_2$ be $2$-edge-connected cosimple graphs with exactly four vertices and without loops. Assume that $L(G_1)=L(G_2)$ and $L(G_1)$ has at least two circuits. Then  at least one of the following holds.
 \begin{itemize}
    \item[(1)] $G_1$ and $G_2$ are $2$-isomorphic;
    \item[(2)] $G_1$ and $G_2$ are isomorphic to $K_4$;
    \item[(3)] $G_1$ and $G_2$ are $2$-isomorphic to the graphs pictured in Figure \ref{Figure 6}.
 \end{itemize}
 \end{lem}

\begin{figure}[htbp]
\begin{center}
\includegraphics[page=5,height=3.5cm]{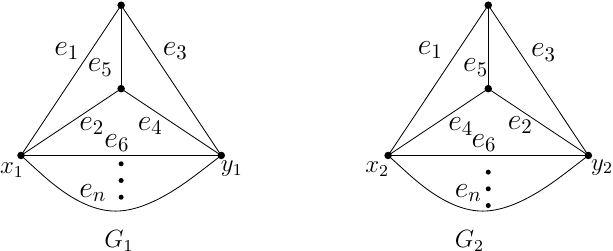}
\caption{$n\geq7$.}
\label{Figure 6}
\end{center}
\end{figure}

\begin{proof}
By Lemma \ref{contract-L} (3), for each edge $e\in E(G_1)$ we have $L(G_1/e)=L(G_2/e)$. If $G_1/e$ and $G_2/e$ are 2-isomorphic for each edge $e\in E(G_1)$, then Lemma \ref{contract-G} implies that $G_1$ and $G_2$ are 2-isomorphic. So we may assume that there is some edge $f\in E(G_1)$ such that $G_1/f$ and $G_2/f$ are not 2-isomorphic. Since $L(G_1)$ has at least two circuits, $L(G_1/f)$ also has at least two circuits. Moreover, since $G_1/f$ and $G_2/f$ are cosimple graphs with exactly three vertices, by Lemma \ref{n<4} we have that {\bf (a)} the graphs $G_1/f$ and $G_2/f$ are isomorphic to $K_3^{m,2,n}$ for some integers $m\in\{1,2\}$ and $n\geq2$;  moreover, when $n\geq3$ the $n$ parallel edges in  $G_1$ are also the $n$ parallel edges in $G_2$. \\

 {\noindent\bf \ref{n=4}.1.} {\sl Two edges are parallel in $G_1$ if and only if they are parallel in $G_2$. }

\begin{proof}[Subproof.]
If two edges are parallel in $G_1$ but not $G_2$, then contracting one of the edges produces a counterexample to Lemma \ref{loop-case}.
\end{proof}

The simple proof of \ref{n=4}.1 is given by the referee. Since no non-trivial parallel classes in $G_1$ or $G_2$ contains $f$ by (a),  \ref{n=4}.1 implies  \\

 {\noindent\bf \ref{n=4}.2.} {\sl Each $2$-edge path joining the end-vertices of a non-trivial parallel class of $G_1$ is also a $2$-edge path joining the end-vertices of the non-trivial parallel class of $G_2$. }\\

{\noindent\bf \ref{n=4}.3.} {\sl Let $P_1,P_2$ be non-trivial parallel classes of $G_1$. Then $\si(G_1|P_1\cup P_2\cup f)$ is a triangle. }

\begin{proof}[Subproof.]
Since $G_1/f$ has no loop, neither $P_1$ nor $P_2$ contains $f$. If $P_1$ and $P_2$ are not contained in a parallel class of $G_1/f$, then $P_1$ and $P_2$ are contained in two different non-trivial parallel classes of $G_1/f$. Moreover, since $P_1$ and $P_2$ are also non-trivial parallel classes of $G_2$ by \ref{n=4}.1, by (a) we have that $G_1/f$ and $G_2/f$ are isomorphic, a contradiction. So $P_1$ and $P_2$ are contained in a parallel class of $G_1/f$. Then $\si(G_1|P_1\cup P_2\cup f)$ is a triangle.
\end{proof}

First we consider the case that $G_1/f$ is isomorphic to $K_3^{2,2,n}$. By \ref{n=4}.3, $G_1$ is obtained from $G_1/f$ by splitting a degree-4 vertex. Since $G_1$ is cosimple, $G_1$ is isomorphic to the graph pictured in Figure \ref{Figure 6} with $e_5$ relabelled by $f$. Let $P$ be the unique non-trivial parallel class of $G_1$ with $n$ edges. Since $P$ is a also non-trivial parallel class of $G_2$ by \ref{n=4}.1 and the fact that $G_2/f$ is isomorphic to $K_3^{2,2,n}$, the graph $G_2$ is isomorphic to the graph pictured in Figure \ref{Figure 6} with $e_5$ relabelled by $f$. So (3) holds.

Secondly we consider the case that $G_1/f$ is isomorphic to $K_3^{1,2,n}$.  Let $e_i$ be the edge of $G_i/f$ that is not in a parallel class for $1\leq i\leq 2$. Evidently, when $n\geq3$, since $G_1/f$ and $G_2/f$ are not 2-isomorphic, $e_1\neq e_2$. Since each vertex of $G_1$ has degree at least three, by \ref{n=4}.3 the graph $G_1$ is obtained from $G_1/f$ by splitting the vertex $v$ incident with two non-trivial parallel classes.  When $|st_{G_1/f}(v)|=4$, since $G_1$ is cosimple $G_1$ is isomorphic to $K_4$. By symmetry $G_2$ is also isomorphic to $K_4$. So (2) holds.

Assume that $|st_{G_1/f}(v)|\geq5$, that is, a non-trivial parallel class $P$ incident with $v$ in $G_1/f$ has at least three edges. Then some proper subset $P'$ of $P$ is a non-trivial parallel class in $G_1$ as $G_1$ is cosmiple. Let $\{f_1,f_2\}$ be the 2-edge parallel class in $G_1/f$. Since $\{f, f_1,f_2\}$ is a cycle in $G_1$ and $\{e_1, f_1,f_2\}$ is the neighbourhood of a degree-3 vertex in $G_1/f$ and $G_1$, by symmetry we may assume that $e_1,f_1$ is a 2-edge path joining the end-vertices of $P'$ in $G_1$ and $f_2$ is not incident with $P'$. On the other hand, by symmetry, $e_2$ is also contained in a 2-edge path joining the end-vertices of $P'$ in $G_2$. So $f_1=e_2$ as $e_2\in\{f_1,f_2\}$, consequently, $|P-P'|=1$, for otherwise there are two such $P'$, which is not possible. Therefore, (3) holds.
\end{proof}

\begin{lem}\label{n=5}
Let $G_1$ and $G_2$ be $2$-edge-connected cosimple graphs with five vertices and without loops. Assume that $L(G_1)=L(G_2)$ and $L(G_1)$ has at least two circuits. Then $G_1$ and $G_2$ are $2$-isomorphic.
 \begin{proof}
By Lemma \ref{contract-L} (3), for each edge $e\in E(G_1)$ we have $L(G_1/e)=L(G_2/e)$. If $G_1/e$ and $G_2/e$ are 2-isomorphic for each edge $e\in E(G_1)$, then Lemma \ref{contract-G} implies that $G_1$ and $G_2$ are 2-isomorphic. So we may assume that for some edge $f\in E(G_1)$ we have $L(G_1/f)=L(G_2/f)$ but $G_1/f$ and $G_2/f$ are not 2-isomorphic.

We claim that $G_1/f$ and $G_2/f$ have no loops. Since $L(G_1/f)$ has at least two circuits, Lemma \ref{loop-case} implies that $loop(G_1/f)=loop(G_2/f)$. If $loop(G_1/f)\neq\emptyset$, then Corollary \ref{loop-iso} implies that $G_1/f$ and $G_2/f$ are 2-isomorphic, a contradiction. 


Since $G_1/f$ and $G_2/f$ are cosimple with four vertices and without loops, Lemma \ref{n=4} implies that $G_1/f$ and $G_2/f$ are either 2-isomorphic to $K_4$ or to the graphs pictured in Figure \ref{Figure 6}. Since each vertex in $K_4$ has degree three and $G_1$ and $G_2$ are cosimple, neither $G_1/f$ nor $G_2/f$ is 2-isomorphic to $K_4$. So $G_1/f$ and $G_2/f$ are 2-isomorphic to the graphs pictured in Figure \ref{Figure 6} with $G_i$ replaced by $G_i/f$ and all other labeling the same.  Let $P$ be the non-trivial parallel class in $G_1/f$ and $G_2/f$. For each $i\in \{1,2\}$, let $u_i$ and $v_i$ be the end-vertices of $f$ in $G_i$, let $x_i$ be the vertex of degree at least four in $G_i/f$ incident with $e_1$, and $y_i$ be the vertex of degree at least four in $G_i/f$ incident with $e_3$.  Since $|st_{G_i}(u_i)|, |st_{G_i}(v_i)|\geq 3$, the graph $G_i$ is obtained from $G_i/f$ by splitting $x_i$ or $y_i$. Without loss of generality we may assume that $G_i$ is obtained from $G_i/f$ by splitting $x_i$ for each $i\in \{1,2\}$.

We claim that $|E(G_1/f)|=7$, that is, $|P|=2$. Assume otherwise. Then there is a subset $P'$ of $P$ with $|P'|\geq 2$ such that $P'$ is also a parallel class in $G_1$. Using a similar analysis to the one in the proof of  \ref{n=4}.1 we have that $P'$ is also a parallel class in $G_2$.  Assume that $e_1,e_2$ are adjacent in $G_1$. Since a union of  any two edges in $P'$ and $\{e_1,e_2,e_5\}$ or $\{e_3,e_4,e_5\}$ is a circuit of $L(G_1)$, we deduce that $\{e_1,e_4,f\}\cup P'$ are contained in $st_{G_2}(u_2)$ or $st_{G_2}(v_2)$. Hence, $|P-P'|\leq1$, implying that $(P-P')\cup\{f\}$ is a bond of $G_2$ with at most two edges, a contradiction as $G_2$ is cosimple. So $e_1,e_2$ are not adjacent in $G_1$. By symmetry we may assume that $st_{G_1}(v_1)=\{e_2,f\}\cup P'$. Since $P'$ is a parallel class of $G_2$ and the union of $\{e_3,e_4,e_5\}$ and any two edges in $P'$ is a circuit of $L(G_1)$, by symmetry we may assume that $\{e_4,f\}\cup P'$ are incident with $v_2$. Hence, $|P-P'|=1$ and $st_{G_1}(u_1)=st_{G_2}(u_2)=(P-P')\cup \{e_1,f\}$. Set $\{e_6\}=P-P'$.  See Figure \ref{Figure 7}. Then $\{e_1,e_2,e_3,e_5,e_6,f\}$ is a circuit of $L(G_1)$ but is not a circuit of $L(G_2)$, a contradiction. So $|E(G_1/f)|=7$. Set $E(G_1/f):=\{e_1,e_2,\cdots,e_7\}$.

Since $G_1$ and $G_2$ are cosimple and $|E(G_1/f)|=7$, we have $|st_{G_i}(u_i)|=|st_{G_i}(v_i)|=3$ for each $i\in \{1,2\}$. By symmetry, there are two cases to consider. First we consider the case $st_{G_1}(u_1)=\{f,e_1,e_2\}$. Since $\{e_1,e_2,e_3,e_4,e_5\}$ is a circuit of $L(G_1)$, by symmetry we can assume $st_{G_2}(u_2)=\{f,e_1,e_4\}$. Then $\{e_2,e_3,e_5,e_6,e_7,f\}$ is a circuit of $L(G_1)$ but is not a circuit of $L(G_2)$, a contradiction.

Secondly consider the case $st_{G_1}(u_1)=\{f,e_1,e_6\}$. Then $\{e_1,e_3,e_4,e_5,e_6\}$ is a circuit of $L(G_1)$. On the other hand, by symmetry and the analysis in the last paragraph we have $\{f,e_1,e_4\}\neq \{N_{G_2}(u_2),N_{G_2}(v_2)\}$. So $\{e_1,e_3,e_4,e_5,e_6\}$ is not a circuit of $L(G_2)$, a contradiction.
 \end{proof}
 \end{lem}

 \begin{figure}[htbp]
\begin{center}
\includegraphics[page=6,height=3.5cm]{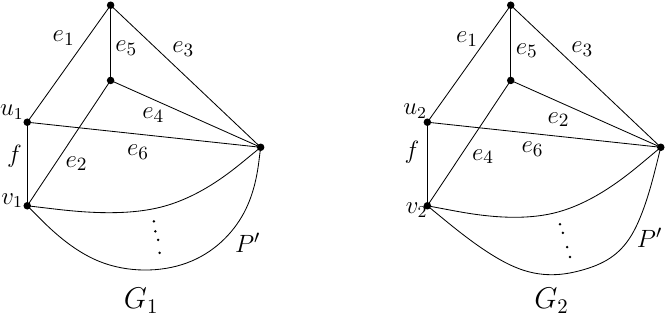}
\caption{}
\label{Figure 7}
\end{center}
\end{figure}

For convenience, Theorem \ref{1-con} is restated here.

\begin{thm}
Let $G_1$ be a $2$-edge-connected graph such that $L(G_1)$ contains at least two circuits. Let $G_2$ be a graph with $L(G_1 )=L(G_2 )$. Then at least one of the following holds.
\begin{itemize}
    \item[(1)] $G_1$ and $G_2$ are $2$-isomorphic.
    \item[(2)] $G_1$ and $G_2$ are $2$-isomorphic to subdivisions of $K_4$, where the edge set of an ear of $G_1$ is also the edge set of an ear of $G_2$.
    \item[(3)] $G_1$ and $G_2$ are $2$-isomorphic to subdivisions of $K_3^{m,2,n}$ for some $m\in\{1,2\}$ and $n\geq2$, where the edge set of an ear of $G_1$ is also the edge set of an ear of $G_2$. Moreover, when $n\geq3$, the $n$ ears in $G_1$ having the same ends also have the same ends in $G_2$. 
    \item[(4)] $G_1$ and $G_2$ are $2$-isomorphic to the graphs pictured in Figure \ref{Figure 1} .
\end{itemize}
\end{thm}

\begin{proof}
 If some loop $e$ of $G_1$ is also a loop of $G_2$, then by Corollary \ref{loop-iso} we have that $G_1\backslash e$ and $G_2\backslash e$ are 2-isomorphic. So $G_1$ and $G_2$ are 2-isomorphic. Moreover, when some link of $G_1$ is a loop of $G_2$, Lemma \ref{loop-case} implies that (4) holds. Therefore, we may assume that neither $G_1$ nor $G_2$ has loops. By Whitney's 2-Isomorphism Theorem we can further assume that $G_1$ and $G_2$ are connected, and each series class of $G_i$ is an ear of $G_i$ for each $i\in \{1,2\}$. Using Lemma \ref{ear-to-ear} we may assume that a subset of $E(G_1)$ is the edge set of an ear of $G_1$ if and only if it is the edge set of an ear of $G_2$. Therefore, we may assume that $G_1$ and $G_2$ are cosimple.

Since the rank of $L(G_i)$ is equal to $|V(G_i)|$, we have $|V(G_1)|=|V(G_2)|$. When $|V(G_1)|\leq 4$, Lemmas \ref{n<4} and \ref{n=4} imply that the result holds. We claim that when $|V(G_1)|\geq 5$ we have that  $G_1$ and $G_2$ are 2-isomorphic. When  $|V(G_1)|=5$, the claim follows from Lemma \ref{n=5}. So we may assume that $|V(G_1)|\geq 6$. For each edge $e\in E(G_1)$, by Lemma \ref{contract-L} (3) we have $L(G_1/e)=L(G_2/e)$. By induction $G_1/e$ and $G_2/e$ are 2-isomorphic. So $G_1$ and $G_2$ are 2-isomorphic by Lemma \ref{contract-G}.
\end{proof}

\section{Acknowledgments}

The authors thank the referee pointing out a short proof to our main result and other improvements.

\end{document}